\documentclass{sl}     

\usepackage{slsec}     
\usepackage{stud_log}
\usepackage{slfoot}
\usepackage{slthm}     
\usepackage{verbatim}                       

\usepackage{amsmath}
\usepackage{amssymb}
\usepackage{enumerate}

\newtheorem{theorem}{Theorem}
\newtheorem{corollary}{Corollary}
\newtheorem{remark}{Remark}
\newtheorem{algorithm}{Algorithm}
\theoremstyle{definition}
\newtheorem{definition}{Definition}
\usepackage{amssymb,amstext,amsmath,url}
\def\C{\mathcal{C}}
\def\F{\mathcal{F}}
\def\K{\mathcal{K}}
\def\M{\mathcal{M}}
\def\G{\mathcal{G}}
\def\H{\mathcal{H}}
\def\I{\mathcal{I}}
\def\P{\mathcal{P}}
\def\S{\mathcal{S}}
\def\pr{\mathbb{P}}

\HeadingsInfo{Author1, Author2}{Conditional Random Quantities and Compounds of Conditionals}

\begin{document}

\setcounter{page}{1}     




\twoAuthorsTitleoneline {Angelo Gilio}{Giuseppe Sanfilippo}{Conditional random quantities and compounds of conditionals}
\HeadingsInfo{A. Gilio, G. Sanfilippo}{Conditional random quantities and compounds of conditionals}




\PresentedReceived{Name of Editor}{October 31, 2012}


\begin{abstract}
In this paper we consider finite conditional random quantities and conditional previsions assessments in the setting of coherence. We use a suitable representation for conditional random quantities; in particular the indicator of a conditional event $E|H$ is looked at as a three-valued quantity with values 1, or 0, or $p$, where $p$ is the probability of $E|H$. We introduce a notion of iterated conditional random quantity of the form $(X|H)|K$ defined as a suitable conditional random quantity, which coincides with $X|HK$ when $H \subseteq K$. Based on a recent paper by S. Kaufmann, we introduce a notion of conjunction of two conditional events and then we analyze it in the setting of coherence. We give a representation of the conjoined conditional and we show that this new object is a conditional random quantity. We examine some cases of logical dependencies, by also showing that the conjunction may be a conditional event; moreover, we introduce the negation of the conjunction and by De Morgan's Law the operation of disjunction. Finally, we give the lower and upper bounds for the conjunction and the disjunction of two conditional events, by showing that the usual probabilistic properties continue to hold. 
\end{abstract}



\section{Introduction}

Probabilistic reasoning under coherence allows a consistent treatment of uncertainty in many applications of statistics, economy, decision theory and artificial intelligence; in particular, it is useful for a flexible numerical approach to inference rules in nonmonotonic reasoning and for the psychology of uncertain reasoning (see, e.g.,  \cite{FPMK11,GiOv12,PfKl06,PfKl09}).  The methods of coherence could also be useful to deepen some theoretical aspects related with the comparison among four well-known nonmonotonic reasoning systems made in \cite{ScTh12} by means of simulations.
In probability theory and in probability logic a relevant problem, largely discussed by many authors, is that of suitably defining logical operations among conditional events. The study of logical operations, such as conjunction and disjunction, among conditionals represents also a basic aspect  in many sectors of artificial intelligence. We recall that a pioneering paper concerning the conjunction, negation and disjunction of conditional events is that one written in 1935 by de Finetti (\cite{deFi35}), where it is proposed a three-valued logic which coincides with that one of Lukasiewicz. An interesting survey of the
contributions by different authors (such as Adams, Belnap, Calabrese, de Finetti, Dubois, van Fraassen, McGee, Goodmann, Lewis, Nguyen, Prade, Schay) to the study of three-valued logics and compounds of conditionals is given in \cite{Miln97}; an extensive study of conditionals has been made in \cite{Edgi95}; see also \cite{McGe89}. Among the many works concerning  logical operations on conditional events we recall for instance \cite{Adam75, BrGi85,Cala87,CaVa99,DuPr91,DuPr94,GoNg88,GoNW91,Scha68}.  A comparison with aspects studied in some of the above papers (with a deepening of the notion of conditional hyperprobability) has been made in \cite{GiSc94}. Logical operations among conditional events have been studied also in \cite{Capo95}, where a generalized notion of atoms for conditional events has been proposed; moreover, a comparison between classical logic and three-valued logic for conditional events has been made in \cite{BaCa96}. As we will show in this paper, the problem of suitably defining logical operations among conditional events has a natural relation with the role of coherence in probabilistic reasoning. In a recent paper by Kaufmann (\cite{Kauf09}) a theory for the compounds of conditionals has been proposed; in this paper we develop a similar theory in the framework of coherence.  In literature the usual approach to the compounds of conditionals  has been that of defining them as suitable
conditionals. In this work,  starting with the paper by Kaufmann, we show that conjunction and disjunction of conditional events in general are not conditional events but {\em conditional random quantities}. Based on the betting scheme of de Finetti (\cite{deFi70}), if we assess $\pr(X|H)=\mu$ for a conditional random quantity $X|H$, then we represent $X|H$ as a numerical quantity which coincides with $X$, or $\mu$, according to whether $H$ is true, or false. In particular, if we assess $P(E|H)=p$ for a conditional event $E|H$, then we represent (the indicator of) $E|H$ as a numerical quantity with set of possible values $\{1,0,p\}$. We recall that the problem of suitably defining the third value for the indicators of conditional events has been carefully examined in many papers by Coletti and Scozzafava (see, e.g., \cite{CoSc02}). Based on the representation of $X|H$, we obtain some results on finite conditional random quantities. Moreover, we give a meaning for the iterated conditional random quantity of the form $(X|H)|K$ as a suitable conditional random quantity which coincides in particular with $X|HK$ when $H \subseteq K$. Then, by exploiting our representation of conditional events, we suitably define the conjunction $(A|H) \wedge (B|K)$ of two conditional events $A|H, B|K$. We show cases of logical dependencies in which the conjunction reduces to a conditional event. Based on the usual definition of negation, we introduce a notion of negation for the conjoined conditional; then, based on De Morgan's Law, we define the disjunction of two conditional events. Finally, by exploiting the methods of coherence, we obtain the lower and upper bounds for the coherent extensions of a probability assessment $(x,y)$ on $\{A|H, B|K\}$ to their conjunction $(A|H) \wedge (B|K)$ and their disjunction $(A|H) \vee (B|K)$. Interestingly, the usual probabilistic properties continue to hold in terms of previsions and this aspect, in our opinion, confirms that the most suitable framework for a right approach to compounds of conditionals is that of conditional random quantities in the setting of coherence. We observe that for the
scope of our paper it is enough to consider finite random quantities.

\section{Preliminary notions and results}
In this section we recall some basic notions and results on
coherence for conditional probability assessments and for conditional prevision assessments.
\subsection{Coherent conditional probability assessments}
\
In our approach an event $A$ represents an uncertain fact described by a (non ambiguous) logical proposition; hence we look at $A$ as a two-valued logical entity which can be true ($T$), or false ($F$).
The indicator of $A$, denoted by the same symbol, is a two-valued numerical quantity which is 1, or 0, according to  whether $A$ is true, or false. The sure event is denoted by $\Omega$ and the impossible event is denoted by $\emptyset$.  Moreover, we denote by $A\land B$ (resp., $A \vee B$) the logical conjunction (resp., logical disjunction). In many cases we simply denote the conjunction between $A$ and $B$ as the product $AB$. By the symbol $A^c$ we denote the negation of $A$. Given any events $A$ and $B$, we simply write $A \subseteq B$ to denote that $A$ logically implies $B$, that is  $AB^c$ is the impossible event $\emptyset$. We recall that $n$ events are logically independent when the number of atoms, or constituents, generated by them is $2^n$. In case of some logical dependencies among the events, the number of atoms is less than $2^n$. Given any events $A$ and $B$, with $A \neq \emptyset$, the conditional event $B|A$ is looked at as a three-valued logical entity which is true (T), or false (F), or void (V), according to whether $AB$ is true, or $AB^c$ is true, or $A^c$ is true. \\
{\em Interpretation with the betting scheme}. We recall that, using the betting scheme of de Finetti (\cite{deFi70}), if you assess $P(B|A)=p$, then you agree to pay an amount $p$, by receiving 1, or 0, or $p$, according to whether $AB$ is true, or $AB^c$ is true, or $A^c$ is true (bet called off). Then, the random gain associated with the assessment $P(B|A)=p$ is $\G = sH(E-p)$, where $s$ is a non zero real number. More in general, let be given a real function $P : \; \mathcal{F} \, \rightarrow \, \mathcal{R}$, where $\mathcal{F}$ is an arbitrary family of conditional events. Given any subfamily $\mathcal{F}_n = \{E_1|H_1, \ldots, E_n|H_n\} \subseteq \mathcal{F}$, the restriction of $P$ to $\mathcal{F}_n$ is the vector $\mathcal{P}_n =(p_1, \ldots, p_n)$, where $p_i =
P(E_i|H_i) \, ,\;\; i = 1, \ldots, n$.
We denote by $\mathcal{H}_n$ the disjunction $H_1 \vee \cdots \vee H_n$.
As $E_iH_i \vee E_i^cH_i \vee H_i^c = \Omega \,,\;\; i = 1, \ldots,
n$,  by expanding the expression
$\bigwedge_{i=1}^n(E_iH_i \vee E_i^cH_i \vee H_i^c)$,
we can represent $\Omega$ as the disjunction of $3^n$ logical
conjunctions, some of which may be impossible.  The remaining ones
are the atoms, or constituents, generated by the family $\mathcal{F}_n$ and, of course, are a partition of $\Omega$. We denote by
$C_1, \ldots, C_m$ the constituents contained in $\mathcal{H}_n$ and (if
$\mathcal{H}_n \neq \Omega$) by $C_0$ the remaining constituent
$\mathcal{H}_n^c =
H_1^c \cdots H_n^c $, so that
\[
\mathcal{H}_n = C_1 \vee \cdots \vee C_m \,,\;\;\; \Omega =
\mathcal{H}_n^c \vee
\mathcal{H}_n = C_0 \vee C_1 \vee \cdots \vee C_m \,,\;\;\; m+1 \leq 3^n
\,.
\]
With $(\mathcal{F}_n, \mathcal{P}_n$) we associate the random gain
${\mathcal{G}} = \sum_{i=1}^n s_iH_i(E_i - p_i)$,
where $s_1, \ldots, s_n$ are $n$ arbitrary real numbers, which is the difference between the amount that you receive, $\sum_{i=1}^n s_i(E_iH_i + p_iH_i^c)$, and the amount that you pay, $\sum_{i=1}^n s_ip_i$. The quantity $\G$ represents the net gain from engaging each transaction $H_i(E_i - p_i)$ at the scale and direction specified by the coefficient $s_i$. Let $g_h$
be the value of $\mathcal{G}$ when $C_h$ is true; of course $g_0 = 0$.
Denoting by $\mathcal{G}_{|\mathcal{H}_n}$ the set of possible values of $\mathcal{G}$ restricted to
$\mathcal{H}_n$, it is $\mathcal{G}_{|\mathcal{H}_n} = \{g_1, \ldots, g_m\}$.
Then, we have
\begin{definition}\label{COER-BET} {\rm The function $P$ defined on $\mathcal{F}$ is coherent
if and only if, for every integer $n$, for every finite sub-family $\mathcal{F}_n$
$\subseteq \mathcal{F}$ and for every $s_1, \ldots, s_n$, one has:
$\min \; \mathcal{G}_{|\mathcal{H}_n} \; \leq 0 \leq \max \;
\mathcal{G}_{|\mathcal{H}_n}$. }\end{definition}
As shown by Definition \ref{COER-BET}, a probability assessment is coherent if and only if, in any finite combination of $n$ bets, it may not happen that the values $g_1, \ldots, g_m$ are all positive, or all negative ({\em no Dutch Book}). \\
Given any integer $n$ we set $J_n=\{1,2,\ldots,n\}$; for  each $h\in J_m$ with the constituent $C_h$  we associate a point
$Q_h = (q_{h1}, \ldots, q_{hn})$, where $q_{hj} = 1$, or 0, or $p_j$, according to whether $C_h \subseteq E_jH_j$, or $C_h \subseteq E_j^cH_j$, or $C_h \subseteq H_j^c$.
Denoting by $\mathcal{I}$ the convex hull of  $Q_1, \ldots, Q_m$,
based on the penalty criterion, it can be proved
(\cite[Thm. 4.4]{Gili90}, see also \cite{Gili92,GiSa11a})

\begin{theorem}\label{CNES}{\rm
The function $P$ is coherent if and only if, for every finite subfamily $ \mathcal{F}_n
\subseteq \mathcal{F}$, one has  $\mathcal{P}_n \in \mathcal{I}$. } \end{theorem}
The condition $\mathcal{P}_n \in \mathcal{I}$ is equivalent to the solvability of the
following system ($\Sigma$) in the unknowns $\lambda_1, \ldots,
\lambda_m$
\[
\Sigma: \hspace{1 cm}
\sum_{h=1}^m q_{hj} \lambda_h = p_j \; , \; \; j\in J_n \,
;\;\; \sum_{h=1}^m \lambda_h = 1 \; ;\; \; \lambda_h \geq 0 \, ,
\; h \in J_m\,.
\]
We say that system $\Sigma$ is  associated with the pair $(\mathcal{F}_n,\mathcal{P}_n)$.
Notice that, by a suitable alternative theorem (\cite[Thm 2.9]{Gale60}), solvability of system $\Sigma$ amounts to  condition $\min  \mathcal{G}_{|\mathcal{H}_n} \leq 0 \leq \max
\mathcal{G}_{|\mathcal{H}_n}$. Hence, Theorem \ref{CNES} provides a geometrical meaning for the notion of coherence given in Definition~\ref{COER-BET}.
\subsection{Coherence Checking}
Given the assessment  $\mathcal{P}_n$ on $\mathcal{F}_n$, let $S$ be the set of solutions $\Lambda = (\lambda_1, \ldots,
\lambda_m)$ of the system $\Sigma$. Then, assuming $S \neq \emptyset$, define
\[\begin{array}{l}
\Phi_j(\Lambda) = \Phi_j(\lambda_1, \ldots, \lambda_m) = \sum_{r :
C_r \subseteq H_j} \lambda_r \; , \; \; \; j \in J_n \,;\; \Lambda \in S \,;
\\
M_j  =  \max_{\Lambda \in S } \; \Phi_j(\Lambda) \; , \; \; \; j\in J_n\,;
\;\;\; I_0  =  \{ j \, : \, M_j=0 \} \,.
\end{array}\]
We observe that, assuming $\P_n$ coherent, each solution $\Lambda=(\lambda_1, \ldots,
\lambda_m)$ of system $\Sigma$ is a coherent extension of the assessment $\mathcal{P}_n$ on $\mathcal{F}_n$ to the family $\{C_1|\H_n,\, \ldots,\,
C_m|\H_n\}$. Then, by the additive property, the quantity $\Phi_j(\Lambda)$ is the conditional probability $P(H_j|\H_n)$ and the quantity $M_j$ is the upper probability $P^*(H_j|\H_n)$ over all the solutions $\Lambda$ of system $\Sigma$.
 Of course, $j \in I_0$ if and only if $P^*(H_j|\H_n)=0$. Notice that $I_0 \subset J_n=\{1, \ldots,
 n\}$. We denote by $(\mathcal{F}_0, \mathcal{P}_0)$ the pair associated with $I_0$.
Given the pair $(\mathcal{F}_n,\mathcal{P}_n)$ and a subset $J \subset J_n$, we denote by $(\mathcal{F}_J, \mathcal{P}_J)$ the pair associated with
$J$ and by $\Sigma_J$ the corresponding system.
We observe that $\Sigma_J$ is solvable if and only if $\mathcal{P}_J  \in \mathcal{I}_J$,
where $\mathcal{I}_J$ is the convex hull associated with the pair $( \mathcal{F}_J,
\mathcal{P}_J)$. Then, we have  (\cite[Thm 3.2]{Gili93}; see also \cite{BiGS03,Gili95})
\begin{theorem}\label{GILIO-93}{\rm
Given a probability assessment $\mathcal{P}_n$ on the family $\mathcal{F}_n$, if
the system $\Sigma$ associated with $(\mathcal{F}_n,\mathcal{P}_n)$ is solvable, then for every $J\subset J_n$, such that $J\setminus I_0\neq \emptyset$, the system $\Sigma_J$ associated with $(\mathcal{F}_J,\mathcal{P}_J)$ is solvable too.}
\end{theorem}
The previous result says that the condition $\P_n \in \I$ implies $\P_J \in \I_J$ when $J\setminus I_0\neq \emptyset$. We observe that, if $\P_n \in \I$, then for every nonempty subset $J$ of $J_n\setminus I_0$ it holds that $J\setminus I_0=J \neq \emptyset$; hence, by Theorem \ref{CNES}, the subassessment $\P_{J_n\setminus I_0}$
on the subfamily $\F_{J_n\setminus I_0}$ is coherent. In particular, when $I_0$ is empty, coherence of $\P_n$ amounts to solvability of system $(\Sigma)$, that is to condition $\P_n \in \I$. When $I_0$ is not empty, coherence of $\P_n$ amounts to the validity of both conditions $\P_n \in \I$ and $\P_0$ coherent, as shown  below (\cite[Thm 3.3]{Gili93}).
\begin{theorem}\label{COER-P0}{\rm The assessment $\mathcal{P}_n$ on $\mathcal{F}_n$ is coherent if and only if the following conditions are satisfied:
(i) $\mathcal{P}_n \in \mathcal{I}$; (ii) if $I_0 \neq \emptyset$, then $\mathcal{P}_0$ is coherent.
}\end{theorem}
\subsection{Coherent conditional prevision assessments}\label{SEC-COER-CPA}

Given an event $H \neq \emptyset$ and a finite random quantity (r.q.) $X$, we denote by $X_{|H}$, the set of possible values of $X$ restricted to $H$ and we set $X_{|H} = \{x_1, x_2, \ldots, x_r\}$. In the setting of coherence, agreeing to the betting metaphor the prevision of $''X$ {\em conditional on} $H''$ (also named $''X$ {\em given} $H''$), $\pr(X|H)$, is defined as the amount $\mu$ you agree to pay, by knowing that you will receive the amount $X$ if $H$ is true, or you will receive back the  amount $\mu$ if $H$ is false (bet called off). 
In what follows we define the conditional random quantity (c.r.q.) $''X$ {\em given} $H''$, denoted by $X|H$, as the amount that you receive when you stipulate a bet on ‘$X$ conditional on $H$’.
Then, it holds that $X|H = XH + \mu H^c$, where $\mu = \pr(X|H)$, so that we can look at the c.r.q. $X|H$ as the {\em unconditional} r.q. $XH + \mu H^c$. We observe that, if $\mu \notin X_{|H}$, then $X|H \in \{x_1, x_2, \ldots, x_r, \mu\}$. Moreover, denoting by $A_i$ the event $(X=x_i),\, i\in J_r$, the family $\{A_1H,\ldots,A_rH,H^c\}$ is a partition of $\Omega$ and we have
\[
X|H = XH + \mu H^c = x_1A_1H + \cdots + x_rA_rH + \mu H^c \,.
\]
In particular, when $X$ is an event $A$, the prevision of $X|H$ is the probability of $A|H$ and, if you assess $P(A|H) = p$, then for the indicator of $A|H$, denoted by the same symbol, we have $A|H = AH + pH^c \in \{1,0,p\}$.
We observe that the choice of $p$ as the value of $A|H$ when $H$ is false has been also considered in some previous works  (\cite{CoSc02,Gili90,Jeff91,Lad96,McGe89,StJe94,vanF76}). One peculiarity of our coherence-based approach is that we avoid ad-hoc (and may be inconsistent) evaluations like $P(A|H)=1$ when $P(H)=0$; in this way, some basic probabilistic formulas, such as $P(H^c|H)=0$ and $P(A|H) + P(A^c|H)=1$, are satisfied in all cases included that one where $P(H)=0$.\\
Given a prevision function $\pr$ defined on an arbitrary family $\K$ of finite
conditional random quantities, let $\F_n = \{X_i|H_i, \, i
\in J_n\}$ be a finite subfamily of $\K$ and $\M_n$ the vector
$(\mu_i, \, i \in J_n)$, where $\mu_i = \pr(X_i|H_i)$ is the
assessed prevision for the conditional random quantity $X_i|H_i$.
With the pair $(\F_n,\M_n)$ we associate the random gain $\G =
\sum_{i \in J_n}s_iH_i(X_i - \mu_i)$. Then, using the {\em betting scheme} of de Finetti, we
have
\begin{definition}\label{COER-RQ}{\rm
The function $\pr$ defined on $\K$ is coherent if and only if, $\forall n
\geq 1$,  $\forall \, \F_n \subseteq \K,\, \forall \, s_1, \ldots,
s_n \in \mathbb{R}$, it holds that: $\min \; \mathcal{G}_{|\mathcal{H}_n} \; \leq 0 \leq \max \;
\mathcal{G}_{|\mathcal{H}_n}$. }\end{definition}
\begin{remark}\label{X-CONSTANT}{\rm
We observe that, for $\K = \{X|H\}$, with $\pr(X|H)=\mu$ and $X_{|H} = \{x_1,\ldots,x_r\}$, by the previous definition we have that $\mu$ is coherent if and only if $min \, X_{|H} \leq \mu \leq max \, X_{|H}$. In particular, if $X_{|H}=\{c\}$, then $X|H = cH + \mu H^c$ and $\mu$ is coherent if and only if $\mu = c$. Of course, for $X=H$ (resp. $X=H^c$) it holds that $\mu=1$ (resp. $\mu=0$) and hence $H|H = 1,\, H^c|H = 0$.
}\end{remark}
Given a family $\F_n = \{X_1|H_1,\ldots,X_n|H_n\}$, for each $i \in J_n$ we denote by $\{x_{i1}, \ldots,x_{ir_i}\}$ the set of possible values for the restriction of $X_i$ to $H_i$; then, for each $i \in J_n$ and $j = 1, \ldots, r_i$, we set $A_{ij} = (X_i = x_{ij})$. Of course, for each $i \in J_n$, the family $\{H_i^c, A_{ij}H_i \,,\; j = 1, \ldots, r_i\}$ is a partition of the sure event $\Omega$. Then,
the constituents generated by the family $\F_n$ are (the
elements of the partition of $\Omega$) obtained by expanding the
expression $\bigwedge_{i \in J_n}(A_{i1}H_i \vee \cdots \vee A_{ir_i}H_i \vee
H_i^c)$. We set $C_0 = H_1^c \cdots H_n^c$ (it may be $C_0 = \emptyset$);
moreover, we denote by $C_1, \ldots, C_m$ the constituents
contained in $\H_n = H_1 \vee \cdots \vee H_n$. Hence
$\bigwedge_{i \in J_n}(A_{i1}H_i \vee \cdots \vee A_{ir_i}H_i \vee
H_i^c) = \bigvee_{h = 0}^m C_h$.
With each $C_h,\, h \in J_ m$, we associate a vector
$Q_h=(q_{h1},\ldots,q_{hn})$, where
\begin{equation}\label{VECTOR-Q}
q_{hi}=\left\{\begin{array}{ll} x_{i1} \,, &  C_h \subseteq
A_{i1}H_i \,, \\
..... & .................. \\
x_{ir_i} \,, & C_h \subseteq
A_{ir_i}H_i \,, \\
\mu_i \,, &  C_h \subseteq H_i^c \,.
\end{array}\right.
\end{equation}
In more explicit terms, for each $j \in
\{1,\ldots,r_i\}$ the condition $C_h \subseteq A_{ij}H_i$ amounts to
$C_h \subseteq A_{i1}^c \cdots A_{i,j-1}^cA_{ij}A_{i,j+1}^c \cdots
A_{ir}^cA_{ir_i}^cH_i$. We observe that the vector $Q_h$ is the value of the random vector $(X_1|H_1,\ldots,X_n|H_n)$ when $C_h$ is true; moreover, if $C_0$ is true, then the value of such a random vector is $\M_n = (\mu_1,\ldots,\mu_n)$. Denoting by $\I_n$ the convex hull of $Q_1, \ldots, Q_m$, the condition  $\M_n\in \I_n$ amounts to the existence of a vector $(\lambda_1,\ldots,\lambda_m)$ such that:
$ \sum_{h \in J_ m} \lambda_h Q_h = \M_n \,,\; \sum_{h \in J_ m} \lambda_h
= 1 \,,\; \lambda_h \geq 0 \,,\; \forall \, h$; in other words, $\M_n\in \I_n$ is equivalent to solvability of the following system $\Sigma$ associated with the pair $(\F_n,\M_n)$, in the
nonnegative unknowns $\lambda_1,\ldots, \lambda_m$,
\begin{equation}\label{SYST-SIGMA}
\Sigma: \hspace{0.5cm}\sum_{h \in J_ m} \lambda_h q_{hi} =
\mu_i \,,\; i \in J_n \,; \; \sum_{h \in J_ m} \lambda_h = 1 \,;\;
\lambda_h \geq 0 \,,\;  \, h\in J_m \,.
\end{equation}
Given a subset $J \subseteq J_n$, we set
$\F_J = \{X_i|H_i \,,\, i \in J\} \,,\;\; \M_J = (\mu_i \,,\, i
\in J) \,;
$ then, we denote by $\Sigma_J$, where $\Sigma_{J_n} = \Sigma$, the system like (\ref{SYST-SIGMA}) associated with the pair $(\F_J,\M_J)$. Then, it can be proved the following (\cite{BiGS08})
\begin{theorem}\label{SYSTEM-SOLV}{ \rm [{\em General characterization of coherence}].
Given a family of $n$ conditional random quantities $\F_n =
\{X_1|H_1,\ldots,X_n|H_n\}$ and a vector $\M_n =
(\mu_1,\ldots,\mu_n)$, the conditional prevision assessment
$\pr(X_1|H_1) = \mu_1, \ldots$, $\pr(X_n|H_n) =
\mu_n$ is coherent if and only if, for every subset $J \subseteq J_n$,
defining $\F_J = \{X_i|H_i \,,\, i \in J\}$, $\M_J = (\mu_i \,,\,
i \in J)$, the system $\Sigma_J$ associated with the pair
$(\F_J,\M_J)$ is solvable. }\end{theorem}
A characterization of coherence of conditional prevision assessments by non dominance with respect to proper scoring rules has been given in \cite{BiGS12}. \\
Given the assessment $\M_n =(\mu_1,\ldots,\mu_n)$ on  $\F_n =
\{X_1|H_1,\ldots,X_n|H_n\}$, let $S$ be the set of solutions $\Lambda = (\lambda_1, \ldots,\lambda_m)$ of the system $\Sigma$ defined in  (\ref{SYST-SIGMA}).  For any given event $A$ and for any vector $\Lambda = (\lambda_1, \ldots,\lambda_m)$  we simply denote by $\sum_{A}\lambda_h$ the quantity $\sum_{h:C_h\subseteq A}\lambda_h$. Then, assuming  the system $\Sigma$  solvable, i.e. $S \neq \emptyset$, we define
\[
\Gamma_0 = \{i : \; max_{\Lambda \in S} \; \sum_{H_i}\lambda_h
> 0\} \,;\;\;
I_0 =J_n \setminus \Gamma_0= \{i : \; max_{\Lambda \in S} \; \sum_{H_i}\lambda_h
= 0\} \,;
\]
\[
\F_0 = \{X_i|H_i \,,\, i \in I_0\}\;;\; \M_0 = (\mu_i \,,\, i \in I_0)\;.
\]  Then, we have  (\cite[Thm 3]{BiGS08})
\begin{theorem}\label{CNES-PREV-I_0-INT}{\rm [{\em Operative characterization of coherence}]
A vector of  prevision assessment ${\M_n} = (\mu_1,\ldots,\mu_n)$ on
the family $\F_n = \{X_1|H_1,\ldots,X_n|H_n\}$ is coherent if
and only if the following conditions are satisfied: \\
(i) the system $\Sigma$ defined in (\ref{SYST-SIGMA}) is solvable ; (ii) if $I_0 \neq \emptyset$, then $\M_0$ is coherent. }
\end{theorem}
\begin{remark} \label{Rem-Gamma0}{\rm
Notice that, if system (\ref{SYST-SIGMA}) is solvable, then it could be proved that the sub-assessment $\M_{\Gamma_0}$ on the subfamily $\F_{\Gamma_0}$ is coherent.}
\end{remark}
Based on Theorem \ref{CNES-PREV-I_0-INT} the following algorithm for coherence checking has been given in \cite[see Remark 2]{BiGS08}.

\begin{algorithm}\label{ALG-PREV-INT}{\rm
Let be given the triplet $(J_n, \F_n, \M_n)$. \\
1. Construct the system $(\ref{SYST-SIGMA})$ and check its
solvability; \\
2. If the system $(\ref{SYST-SIGMA})$ is not solvable then
$\M_n$ is not g-coherent and the procedure stops, otherwise
compute the set $I_0$; \\
3. If $I_0 = \emptyset$ then $\M_n$ is g-coherent and the
procedure stops, otherwise set $(J_n, \F_n, \M_n) = (I_0, \F_0,
\M_0)$ and repeat steps 1-3. }
\end{algorithm}
\section{Some results on conditional random quantities}
\label{SEC:CPT}
We first deepen some aspects on conditional random quantities; then, by also exploiting linearity of prevision, we give a simple proof of the general compound prevision theorem.
\begin{theorem}\label{PREL-SUM}{\rm Given any quantity $a$, any event $H \neq \emptyset$ and any random quantities $X$ and $Y$, we have
\begin{equation}\label{XY-SUM}
(i) \;\; (aX)|H = a(X|H) \,,\;\; (ii) \;\; X|H + Y|H = (X + Y)|H \,.
\end{equation}
}\end{theorem}
\begin{proof}
(i) We set $\pr(X|H) = \mu$, so that $\pr[(aX)|H] = a\mu$; then, $(aX)|H = aXH + a\mu H^c = a (X|H)$ and we can simply write $aX|H$. \\
(ii)We set $\pr(X|H) = \mu,\, \pr(Y|H) = \nu,\, \pr[(X+Y)|H] = \eta$, with $(\mu,\nu,\eta)$ coherent; then, we have
\[
X|H = XH + \mu H^c \,,\; Y|H = YH + \nu H^c \,,\; (X+Y)|H = XH + YH + \eta H^c \,.
\]
It follows
\[
X|H + Y|H - (X+Y)|H = 0 \cdot H + (\mu+\nu-\eta)H^c \,;
\]
hence, $X|H + Y|H - (X+Y)|H$ is a c.r.q. $Z|H$, with $Z_{|H} = \{0\}$.
Then, by Remark \ref{X-CONSTANT}, $\mu+\nu-\eta = 0$ and hence
$X|H + Y|H = (X+Y)|H$.
\end{proof}
By the same reasoning, it follows: $aX|H + bY|H = (aX+bY)|H$.
\begin{theorem}\label{EQ-CRQ}{\rm Given two c.r.q.'s $X|H, Y|K$, with $\pr(X|H)=\mu,\, \pr(Y|K)=\nu$ and with $(\mu,\nu)$ coherent, assume that $X|H = Y|K$ when the disjunction $H \vee K$ is true. Then $X|H = Y|K$.
}\end{theorem}
\begin{proof} We observe that $X|H = XH + \mu H^c = XH + \mu H^cK + \mu H^cK^c$ and $Y|K = YK + \nu K^c = YK + \nu HK^c + \nu H^cK^c$; then, the hypothesis amounts to the equality: $XH + \mu H^cK = YK + \nu HK^c$. Now, let $C_1,\ldots,C_m$ be the constituents contained in $H \vee K$ and $Q_1,\ldots,Q_m$ be the corresponding points associated with $(\{X|H, Y|K\}, (\mu,\nu))$. For each $Q_h = (q_{h1},q_{h2})$ it holds that $q_{h1}=q_{h2},\; h\in J_m$. Moreover, by coherence, the point $(\mu,\nu)$ belongs to the convex hull of $Q_1,\ldots,Q_m$; then, it follows $\mu=\nu$; therefore $X|H = XH + \mu H^cK + \mu H^cK^c = Y|K$.
\end{proof}
By the previous result it immediately follows
\begin{corollary}\label{PREL-CPT}{\rm Given any event $H \neq \emptyset$ and any random quantities $X$ and $Y$, we have
\begin{equation}\label{2RQ}
XH = YH \; \Longrightarrow \; X|H = Y|H \,.
\end{equation}
}\end{corollary}
Given any c.r.q.'s $X|H$ and $Y|K$, with $\pr(X|H)=\mu,\, \pr(Y|K)=\nu$ and with $(\mu,\nu)$ coherent, we have
$X|H + Y|K = XH + \mu H^c + YK + \nu K^c$; then we obtain
\begin{theorem}\label{SUM-PREV}{\rm Given any c.r.q.'s $X|H$ and $Y|K$, we have
\begin{equation}\label{SUM-FORMULA}
\pr(X|H + Y|K) = \pr(X|H) + \pr(Y|K) \,.
\end{equation}
}\end{theorem}
\begin{proof}
By linearity of prevision, we have
\[
\pr(X|H + Y|K) = \pr(XH + \mu H^c + YK + \nu K^c) =
 \]
 \[
 = \pr(XH + \mu H^c) + P(YK + \nu K^c) = \pr(X|H) + \pr(Y|K) \,.
\]
\end{proof}
We recall that, agreeing to the betting metaphor, the prevision $\mu$ for a c.r.q. $X|H$ is what should be payed in order to receive the amount $X|H$; then by linearity of prevision
\[
 \pr(X|H) = \mu = x_1P(E_1H) + \cdots + x_nP(E_nH) + \mu P(H^c) = \pr(XH) + \mu P(H^c) \,
 \]
from which it follows: $\pr(XH)  = P(H)\pr{(X|H)}$. More in general, we have
\begin{theorem}\label{COMP-PREV-TH}{\rm Given two events $H \neq \emptyset, K \neq \emptyset$ and a r.q. $X$, let $(x,y,z)$ be a coherent assessment  on $\{H|K, X|HK, XH|K\}$. Then: (i) $X|HK = (XH + yH^c)|K$; (ii) $z=xy$; that is: $\pr(XH|K) = P(H|K)\pr(X|HK)$.
}\end{theorem}
\begin{proof}
(i) First of all we observe that $HK \vee K = K$; moreover
\[
X|HK = XHK + y(HK)^c = XHK + yK^c + yH^cK = (XH + yH^c)K + yK^c \,,
\]
and, by setting $\pr[(XH + yH^c) | K] = \mu$, we have $(XH + yH^c) | K = $
\[
= (XH + yH^c)K + \mu K^c = XHK + yH^cK + \mu K^c = X|HK + (\mu-y)K^c \,.
\]
Hence $X|HK = (XH + yH^c) | K$ for $K=1$. Then, by Theorem \ref{EQ-CRQ}, $y=\mu$ and $X|HK = (XH + yH^c)|K$. \\
(ii) By linearity of prevision:
\[
y = \pr(X|HK) = \pr[(XH + yH^c)|K] = \pr(XH|K) + yP(H^c|K) = z + y(1 - x) \,;
\]
hence: $z=xy$, which represents the {\em general compound prevision theorem}.
\end{proof}
We observe that, given two r.q.'s $X|H, Y|H$, from condition (ii) in Theorem \ref{COMP-PREV-TH}, we have
\begin{small}
\[
\pr(XH|H) = P(X|H)P(H|H)=\pr(X|H),\, \pr(YH|H) = P(Y|H)P(H|H)=\pr(Y|H)
\]
\end{small}
and, assuming $XH=YH$, it follows $\pr(Y|H) = \pr(YH|H) = \pr(XH|H) = \pr(X|H)$; hence $X|H = Y|H$, which is the result given in Corollary \ref{PREL-CPT}. \\
We now give a definition for the symbol $(X|H)|K$; we will show that its meaning is different from $X|HK$. We recall that $X|H = XH + xH^c$, where $x = \pr(X|H)$.
\begin{definition}\label{ITER-RQ}{\rm
Given any events $H,K$, with $H \neq \emptyset, K \neq \emptyset$, and a finite r.q. $X$, with $x = \pr(X|H)$, we define $(X|H)|K = (XH + xH^c)|K$.
}\end{definition}
We have two remarks: \\
(a) In condition (i) of Theorem \ref{COMP-PREV-TH} the value $y$ is (not the prevision $x$ of $X|H$ but) the prevision of $X|HK$; hence
\[
X|HK = (XH + yH^c)|K \; \neq \; (X|H)|K = (XH + xH^c)|K \,.
\]
In other words, $X|HK \neq (X|H)|K$; but, under the hypothesis $H \subseteq K$, we have $X|HK = X|H$ and $y = \pr(X|HK) = \pr(X|H) = x$; then
\begin{equation}\label{XCOND-ITER}
X|H = X|HK = (XH + yH^c)|K = (XH + xH^c)|K = (X|H)|K \,.
\end{equation}
We observe that, given any events $A,H,K$, we have $A|HK \neq (A|H)|K$; therefore,  in agreement with \cite{Adam75,Kauf09}, in our approach the Import-Export Principle of McGee (\cite{McGe89}) does not hold. To illustrate by an example that the Import-Export Principle is not valid in general, assume that $K=H^c\vee A$, which is the material conditional associated with  $A|H$; moreover assume that $AH=\emptyset$, so that $P(A|H)=0$. Then the Import-Export Principle cannot be applied because  $A|HK=A|AH=A|\emptyset$; on the contrary, as $H^c\vee A=H^c$, by Definition \ref{ITER-RQ} we have
\[
(A|H)|K = (A|H)|(H^c \vee A) = (A|H)|H^c = (AH + 0\cdot H^c)|H^c = 0|H^c =0;
\]
therefore  (in agreement with the intuition)	 $\pr[(A|H)|K]=P(A|H)=0$, while $P(H^c\vee A)$ could be high. A probabilistic analysis of constructive and non-constructive inferences from the material conditional $A\vee B$ to the associated conditional $B|A^c$ has been given in \cite{GiOv12}.
\\
(b) Given any c.r.q.'s $X|H, Y|K$, and a coherent assessment $\pr(X|H) = x$, $\pr(Y|K) = y$, as $H = H(H \vee K), K = K(H \vee K)$, we have
\begin{equation}\label{XgHvK}
\begin{array}{l}
X|H = X|H(H \vee K) =(X|H)|(H\vee K) = (XH + x H^c)|(H \vee K) \,,
\\
Y|K = Y|K(H \vee K)=(Y|K)|(H\vee K) = (YK + y K^c)|(H \vee K) \,.
\end{array}
\end{equation}
Then, by (\ref{XY-SUM}), we obtain
\begin{equation}\label{SUM-GEN}
X|H + Y|K = (XH + x H^c + YK + y K^c)|(H \vee K) \,,
\end{equation}
which shows that $X|H + Y|K$ coincides with the c.r.q. $Z|(H \vee K)$, where $Z=XH + x H^c + YK + y K^c$, with $\pr[Z|(H \vee K)] = x + y$.
\section{Conjunction of conditional events}
\label{SEC:CONJ}
Some authors look at the conditional ``if $A$ then $C$'', denoted $A \rightarrow C$ , as the event $A^c \vee C$ ({\em material conditional}), but since some years it is becoming standard to look at $A \rightarrow C$ as the conditional event $C|A$ (see e.g. \cite{GiOv12,PfKl10}).  A theory of the compounds of conditionals is a not easy and controversial topic of research; it has been studied by many researchers in many fields, such as mathematics, philosophical logic, artificial intelligence, nonmonotonic reasoning, psychology. A very general discussion of the different aspects which concern conditionals has been given in \cite{Edgi95,Miln97}.
\subsection{Compounds of conditionals in the approach of Kaufmann}
The probabilistic theory of conditionals proposed in \cite{Kauf09} is based on the model theory proposed in \cite{vanF76} and on the assignment of truth values to complex conditionals suggested in  \cite{StJe94}.
In particular, Kaufmann uses the notion of {\em Stalnaker Bernoulli space} to build a complex procedure by means of which probabilistic formulas are obtained which suggest how to assign values to conditionals. To  illustrate such a procedure,  consider a conditional $A \rightarrow C$ and the associated conditional event $C|A$, with $P(A)>0$, so that $P(C|A)=\frac{P(AC)}{P(A)}$. Now, let us consider an infinite sequence of pairs of events $(A_1, C_1) \,,\; (A_2, C_2) \,,\; \ldots (A_n, C_n) \,,\; \ldots$,
with the events in each pair stochastically independent from the events in the other ones, and
with $P(A_i) = P(A),\, P(A_iC_i)=P(AC),\; \forall i$, so that $P(C_i|A_i) = P(C|A),\, \forall i$. In the approach of Kaufmann, in order to assign the value to $A \rightarrow C$, the pairs in the sequence are observed {\em until the first time} the antecedent, say $A_i$, is true; then the value of the consequent, $C_i$, is assigned to $A \rightarrow C$. In other words, by considering the partition of $\Omega$ obtained by expanding the expression
\[
(A_1C_1 \vee A_1C_1^c \vee A_1^c) \wedge \cdots \wedge (A_nC_n \vee A_nC_n^c \vee A_n^c) \wedge \cdots =
 H_0 \vee H_1 \vee H_2 \vee H_3 \,,
\]
where $H_1 = A_1C_1,\; H_2 = \bigvee_{n=2}^{\infty} \, A_1^c \cdots A_{n-1}^cA_nC_n,\;
H_3 = A_1^cA_2^c \cdots A_n^c \cdots$, $H_0 = H_1^cH_2^cH_3^c$.
By definition, in the paper of Kaufmann it holds that:
\[
(A \rightarrow C) \wedge A = 1 \Longleftrightarrow H_1 = 1;\;
(A \rightarrow C) \wedge A^c = 1 \Longleftrightarrow H_2 = 1;
\]
hence $(A \rightarrow C) = 1 \Longleftrightarrow H_1 \vee H_2 = 1$.
Then, in the Stalnaker Bernoulli space a probability $P^*$ is constructed such that
\[
P^*[(A \rightarrow C) \wedge A] = P(H_1) = P(A_1C_1) = P(AC) \,;
\]
\[ P^*[(A \rightarrow C) \wedge A^c] = P(H_2) =
P(A_1^c)P(A_2C_2) + P(A_1^c)P(A_2^c)P(A_3C_3) + \cdots =
\]
\[
= P(AC)P(A^c)[1 + P(A^c) + \cdots + [P(A^c)]^n + \cdots] = \frac{P(AC)P(A^c)}{P(A)} \,;
\]
\[
P^*(A \rightarrow C) = P^*[(A \rightarrow C) \wedge A] + P^*[(A \rightarrow C) \wedge A^c] =
P(H_1) + P(H_2) =
\]
\[
= P(AC) + \frac{P(AC)P(A^c)}{P(A)} = \frac{P(AC)[P(A) + P(A^c)]}{P(A)} = \frac{P(AC)}{P(A)} = P(C|A) \,.
\]
Then, as suggested by the result above, Kaufmann shows that, by defining the truth value of $A \rightarrow C$ as:
\[
V(A \rightarrow C) = \left\{
\begin{array}{ll}
1, & AC \;\; true \\
0, & AC^c \;\; true \\
P(C|A), & A^c \;\; true
\end{array}\right.\] \\
it follows: $P(A \rightarrow C) = P(C|A)$. With the approach of Kaufmann, the conditional $A \rightarrow C$ is indeterminate if and only if $H_3$ is true, which has probability
$P(H_3) = P(A_1^c) \cdots P(A_n^c) \cdots = P(A^c) \cdots P(A^c) \cdots = 0$. \\
By a similar reasoning, for the conjoined conditional $(A \rightarrow B) \wedge (C \rightarrow D)$, assuming $P(A \vee C) > 0$, Kaufmann obtains the formula
\[\begin{array}{l}
P[(A \rightarrow B) \wedge (C \rightarrow D)] = \frac{P(ABCD) + P(B|A) P(A^cCD) + P(D|C) P(ABC^c)}{P(A \vee C)} \,.
\end{array}\]
Based on this result, Kaufmann suggests a natural way of defining the values of conjoined conditionals.
\begin{remark}{\rm In the setting of coherence, if $P(C|A)=z$, then $C|A=AC + zA^c$ and, assuming $P(A)>0$, by linearity of prevision (by iteratively replacing $z$) we obtain
\[
\begin{array}{l}
z = P(AC) + zP(A^c) = P(AC) + P(A^c)[P(AC) + zP(A^c)] =\\
= P(AC)[1 + P(A^c)] + zP(A^c)^2 =  \cdots =
\\
= P(AC)[1 + P(A^c) + \cdots + P(A^c)^{n-1}] + zP(A^c)^n = \cdots =
\\
= P(AC)[1 + P(A^c) + \cdots + P(A^c)^{n-1} + \cdots] =
 P(AC) \cdot \frac{1}{P(A)} = \frac{P(AC)}{P(A)} \,.
\end{array}
\]
}\end{remark}
The previous iterative scheme can be associated to a sequence of conditional bets, where the bet is repeated each time it is called off; the process ends the first time the bet is not called off.
\subsection{Some critical comments}
We think that the approach of Kaufmann opens an interesting perspective; it produces very nice results and preserves, as we will show, well known probabilistic properties which hold in the classical setting.
At the same time, we think that, in the setting of coherence, we can obtain (and we can generalize) such results in a direct and simpler way. As a first comment, we observe that to avoid ambiguities in his construction Kaufmann should refer to the sequence of pairs of events
$(A_1, C_1) \,,\; (A_2, C_2) \,,\; \ldots,\; (A_n, C_n) \,,\; \ldots$,
and not simply to the pair $(A,C)$. The iterative procedure introduced by Kaufmann can also be used in our approach: \\
- given any integer $n$, let us consider the pairs
$(A_1, C_1), (A_2, C_2), \ldots (A_n, C_n)$,
and the partition $\{H_0,\ldots,H_3\}$, where
\[
H_1 = A_1C_1,\, H_2 = \bigvee_{k=2}^{n} \, A_1^c \cdots A_{k-1}^cA_kC_k,\, H_3=A_1^c \cdots A_n^c,\, H_0= H_1^cH_2^cH_3^c;
\]
- then, let us consider the conditional event $E|K$, where $E = H_1 \vee H_2$ and $K = H_0 \vee H_1 \vee H_2 = A_1 \vee \cdots \vee A_n$, with $P(E|K) = z$. We have
\[
E|K = \left\{\begin{array}{ll}
1, & H_1 \vee H_2 = 1 \,, \\
0, & H_0 = 1 \,, \\
z, & H_3 = 1 \,.
\end{array}\right.
\]
We recall that $P(A_i)=P(A) > 0, P(A_iC_i)=P(AC), i=1,\ldots,n$; then
\[
\begin{array}{l}
P(H_1) = P(AC) \,,\; P(H_2) = P(AC)[P(A^c) + \cdots + P(A^c)^{n-1}] \,,\\
P(H_1 \vee H_2) = P(H_1) + P(H_2) = P(AC)[1 + P(A^c) + \cdots + P(A^c)^{n-1}] =\\
= P(AC) \frac{1-P(A^c)^n}{1 - P(A^c)} = P(C|A) [1-P(A^c)^n] \,,\\
P(H_3) = P(A_1^c \cdots A_n^c) = P(A^c)^n \, \underset{n \rightarrow \infty}{\longrightarrow} \, 0 \,.
\end{array}
\]
Therefore
\[\begin{array}{l}
P(E|K) = z = 1 \cdot P(H_1 \vee H_2) + 0 \cdot P(H_0) + z \cdot P(H_3) = \\
= P(C|A) [1-P(A^c)^n] + z P(A^c)^n \,;
\end{array}\]
hence: $z[1 - P(A^c)^n] = P(C|A) [1-P(A^c)^n]$; so that: $z = P(C|A),\, \forall \, n$. \\
The probability of the event $''E|K \; true$ or $E|K \; false''$
tends to 1, when $n \rightarrow \infty$, and the probability $z$ is constant, $z=P(C|A),\; \forall n$. \\
{\em A basic aspect}: if we only assess $P(B|A)=x, P(D|C)=y$, how can we check the consistency of the extension $P[(A \rightarrow B) \wedge (C \rightarrow D)]=z$ ? \\
In our setting $(A \rightarrow B) \wedge (C \rightarrow D)$ is looked at as a {\em conditional random quantity} $(B|A) \wedge (D|C)$; hence, we speak of {\em previsions} (and not of probabilities) of conjoined conditionals.  Moreover, we can manage without problems the case $P(A \vee C) = 0$ and, by starting with the assessment $P(B|A) = x,\; P(D|C) = y$, we can determine the values $z = \pr[(B|A) \wedge (D|C)]$ which are coherent extensions of $(x,y)$ on $\{B|A,D|C\}$.
\subsection{Conjunction of conditionals in the setting of coherence}
We introduce the notion of conjunction, by first giving some logical and probabilistic remarks. Given any events $A, B, H$, with $H \neq \emptyset$, let us consider the conjunction $AB$, or the conjunction $(A|H) \wedge (B|H) = AB|H$. In terms of indicators we have
\[
AB = min \, \{A, B\} = A \cdot B \,,\; AB|H = min \, \{A, B\}|H = (A \cdot B)|H \,;
\]
moreover, if we assess $P(A|H) = x, P(B|H) = y$, then
\[
A|H = AH +xH^c = \left\{\begin{array}{l}
A, \mbox{ if } H=1, \\
x, \mbox{ if } H=0,
\end{array}\right.
 B|H=BH + yH^c = \left\{\begin{array}{l}
B, \mbox{ if } H=1, \\
y, \mbox{ if } H=0.
\end{array}\right.
\]
As we see, {\em conditionally on $H$ being true}, i.e. $H=1$, we have:
\[
AB|H = min \, \{A, B\}|H = min \, \{A|H, B|H\}|H \in \{0,1\} \,.
\]
We set $Z = min \, \{A|H, B|H\} = min \, \{AH +xH^c, BH +yH^c\}$; we have $Z \in \{1,0,x,y\}$ and, defining $\pr(Z|H) = z$, we have $Z|H = ZH + z H^c$, with $Z|H \in \{1,0,z\}$. We observe that $ZH = ABH$; then, by Corollary \ref{PREL-CPT}, we have $Z|H = AB|H$. In other words, $min \, \{A|H, B|H\}|H$ and $AB|H$ are the same conditional random quantity. \footnote{In particular, for $B=A$, in agreement with Definition \ref{ITER-RQ} we have $Z = A|H,\, Z|H = (A|H)|H = A|H,\, z=x$; the equality $(A|H)|H = A|H$ still holds from the viewpoint of iterated conditionals introduced in \cite{GiSa13a}.} Then
\begin{equation}\label{PRE-CONJ}
(A|H) \wedge (B|H) = min \, \{A|H, B|H\}\,|\, H = min \, \{A|H, B|H\}\,|\, (H \vee H) \,.
\end{equation}
Based on formula (\ref{PRE-CONJ}), we introduce below the notion of conjunction among  conditional events.
\begin{definition}[Conjunction]\label{CONJUNCTION}{\rm Given any pair of conditional events $A|H$ and $B|K$, with $P(A|H) = x, P(B|K) = y$, we define their conjunction as
\[
(A|H) \wedge (B|K) = min \, \{A|H, B|K\} \,|\, (H \vee K) \,.
\]
}\end{definition}
\noindent Notice that, defining $Z = min \, \{A|H, B|K\}$, the conjunction $(A|H) \wedge (B|K)$ is the c.r.q. $Z \,|\, (H \vee K)$. Moreover, defining $T = (A|H) \cdot (B|K)$, by Corollary \ref{PREL-CPT} it holds that $Z \,|\, (H \vee K)=T \,|\, (H \vee K)$, while $Z \neq T$. Then, we have
\begin{equation}\label{CONJ-PROD}
(A|H) \wedge (B|K) = (A|H) \cdot (B|K) \,|\, (H \vee K) \,.
\end{equation}
\noindent {\em Interpretation with the betting scheme.}
By assessing $\pr[(A|H) \wedge (B|K)] = z$, you agree to pay the amount $z$ by receiving the amount
$ min \, \{A|H, B|K\}$ if $H \vee K = 1$, or the amount $z$ if the bet is {\em called off} $\;\;(H \vee K = 0)$. That is, you pay $z$, by receiving the amount
\[
(A|H) \wedge (B|K)=
\left \{\begin{array}{ll}
1, & AHBK = 1, \\
0, & A^cH \vee B^cK = 1, \\
x, & H^cBK = 1, \\
y, & AHK^c = 1, \\
z, & H^cK^c = 1;
\end{array}\right.
\]
therefore, {\em operatively}, for $(A|H) \wedge (B|K)$ we obtain the  {\em representation}
\begin{equation}\label{CONJ-REPR}
(A|H) \wedge (B|K) \; = \; 1 \cdot AHBK + x \cdot H^cBK + y \cdot AHK^c + z \cdot H^cK^c \,.
\end{equation}
Then, by {\em linearity} of prevision, it follows
\[
\pr[(A|H) \wedge (B|K)] = z
= P(AHBK) + x P(H^cBK) + y P(AHK^c) + z P(H^cK^c),
\]
and we obtain:
$zP(H \vee K) = P(AHBK) + x P(H^cBK) + y P(AHK^c)$. \\
In particular, if $P(H \vee K) > 0$,  we obtain  the result of Kaufmann
\[
\pr[(A|H) \wedge (B|K)] = \frac{P(AHBK) + P(A|H) P(H^cBK) + P(B|K) P(AHK^c)}{P(H \vee K)} \,.
\]
{\em Some particular cases.}\label{SEC:PART-CASE}
We examine below the conjunction of $A|H$ and $B|K$ for special assessments $(x,y)$ on $\{A|H,B|K\}$ and/or when there are some logical dependencies among $A, B, H, K$. We set $P(A|H)=x,\, P(B|K)=y,\, \pr[(A|H) \wedge (B|K)] = z$.
\begin{enumerate}
\item If $x=y=1$, then $(A|H) \wedge (B|K) = 1 \cdot AHBK + 1 \cdot H^cBK + 1 \cdot AHK^c +$ \linebreak $+ z \cdot H^cK^c =(AH \vee H^c) \wedge (BK \vee K^c)|(H \vee K) = \C(A|H,B|K)$, where $\C(A|H,B|K)$ is the \textit{quasi conjunction}  of $A|H$ and $B|K$.
\item \label{COMP-TH} $K = AH$.
From (\ref{CONJ-PROD}) we have
\[
\begin{array}{l}
(A|H) \wedge (B|AH) = [(A|H) \cdot (B|AH)]|H =\\
= [(AH + xH^c)(ABH + y(AH)^c)]|H = (ABH + xyH^c)|H =\\
= ABH|H + xyH^c|H = AB|H = \C(A|H,B|AH) \,.
\end{array}
\]
Then
\begin{small}
\[
\pr[(A|H) \wedge (B|AH)] = P(AB|H) = P(A|H)P(B|AH) = \pr(A|H)\pr(B|AH)\,,
\]
\end{small}
that is: $z=xy$. On the other hand
\[
AB|H = ABH + zH^c = ABH + xyH^c = (A|H)\cdot(B|AH) \,;
\]
therefore $\pr[(A|H) \cdot (B|AH)] = \pr(A|H)\pr(B|AH)$, which means, as discussed in \cite{GiSa13a}, that $A|H$ and $B|AH$ are uncorrelated (see the next case). In particular, for $H=\Omega$, we have $(B|A)\wedge A = AB$ and
\[
\pr[(B|A)\wedge A]=\pr[(B|A)\cdot A]=P(AB)=P(B|A)P(A) \,.
\]
\item Let be given any conditional events $A|H, B|K$, with $HK=\emptyset$ and with $P(A|H)=x, P(B|K)=y$. As $H$ and $K$ are logically incompatible, the assessment $(x,y)$ is coherent for every $(x,y) \in [0,1]^2$; moreover, it can be verified that the assessment $\pr[(A|H) \wedge (B|K)] = z$ is a \linebreak coherent extension of $(x,y)$ if and only if $z=xy$. We have \linebreak
    $(A|H) \cdot (B|K) = (AH + xH^c)(BK + yK^c) = xH^cBK + yAHK^c + xyH^cK^c$; moreover
    \[
    (A|H) \wedge (B|K) = (A|H) \cdot (B|K) \,|\, (H \vee K) = xH^cBK + yAHK^c + zH^cK^c.
     \]
     From $z=xy$, it follows $(A|H) \wedge (B|K)= (A|H) \cdot (B|K)$; that is, the conjunction is the {\em product} of the {\em conditional random quantities} $A|H, B|K$. Therefore
     \begin{equation}\label{IND-INC}
\pr[(A|H) \cdot (B|K)] = P(A|H)P(B|K) = \pr(A|H)\pr(B|K) \,;
\end{equation}
that is, the {\em prevision of the product} coincides with {\em the product of previsions},
which means that, under the hypothesis $HK=\emptyset$, the random quantities $A|H, B|K$ are uncorrelated. As discussed in \cite{GiSa13a}, the equality (\ref{IND-INC}) does not mean that $A|H$ and $B|K$ are stochastically independent.
\item We recall that from $A \subseteq B$, it follows $AB = A$. This property {\em still holds for conditional events}; that is, under the hypothesis $A|H \subseteq B|K$, where the symbol $\subseteq$ denotes the well known inclusion relation of Goodman and Nguyen (\cite{GoNg88}), we can verify that  $(A|H) \wedge (B|K) = A|H$.  Indeed,  the relation  $A|H \subseteq B|K$ amounts to $AH \subseteq BK$ and $B^cK \subseteq A^cH$ and coherence requires $x \leq y$.  In terms of indicators the inclusion relation implies $A|H \leq B|K$; hence $min \, \{A|H, B|K\} = A|H$. Then, by (\ref{XgHvK}), $A|H = A|H(H \vee K) = (AH +xH^c)|(H \vee K)=(A|H)|(H \vee K) $; hence
\[
\begin{small}
(A|H) \wedge (B|K) =min \, \{A|H, B|K\} |(H\vee K)=(A|H)|(H \vee K) = A|H.
\end{small}
\]
We remark that $A|H \wedge B|K = A|H$ does not imply $A|H \subseteq B|K$; for instance, given any events $H,B,K$, with $H^cB^cK \neq \emptyset$, it holds that $H^c|H \wedge B|K = H^c|H$, but $H^c|H \nsubseteq B|K$; in fact, if $H^cB^cK$ is true, then $B|K$ is false, while $H^c|H$ is void.
\end{enumerate}
\section{Lower and upper bounds for $(A|H) \wedge (B|K)$}
We will now determine the coherent extensions of the assessment $(x,y)$ on $\{A|H, B|K\}$ to the conjunction $(A|H) \wedge (B|K)$.
We recall that the extension $z = P(AB|H)$ of the assessment $(x,y)$ on $\{A|H, B|H\}$, with $A,B,H$ logically independent, is coherent if and only if: $max\{x+y-1,0\} \leq z \leq min\{x,y\}$.
\\ {\em The next theorem show that the same results holds for $(A|H) \wedge (B|K)$}.
 \begin{theorem}{\rm
 Given any coherent assessment $(x,y)$ on $\{A|H, B|K\}$, with $A,H,B,K$ logically independent, and with $H \neq \emptyset, K \neq \emptyset$, the extension $z = \pr[(A|H) \wedge (B|K)]$ is coherent if and only if  the Fr\'echet-Hoeffding bounds are satisfied, that is
\begin{equation}\label{LOW-UPPER}
max\{x+y-1,0\} = z' \; \leq \; z \; \leq \; z'' = min\{x,y\} \,.
\end{equation}
}\end{theorem}
\begin{proof}
First of all we observe that, by logical independence of $A,H,B,K$, the assessment $(x,y)$ is coherent for every  $(x,y)\in [0,1]^2$. We will determine the values $z', z''$ by the geometrical approach described in Subsection \ref{SEC-COER-CPA}. The constituents associated with the family
$\F =  \{A|H,\, B|K,\, (A|H) \wedge (B|K)\}$ and contained in $H \vee K$ are
\[\begin{array}{llll}
 C_1=AHBK, &  C_2=AHB^cK, & C_3=A^cHBK, & C_4=A^cHB^cK,\\
 C_5=AHK^c, & C_6=A^cHK^c, & C_7=H^cBK, & C_8=H^cB^cK.
\end{array}
\]
The associated  points $Q_h$'s are
\[
\begin{array}{llll}
Q_1 = (1,1,1) \,, &  Q_2 = (1,0,0) \,, & Q_3 = (0,1,0) \,, &  Q_4 = (0,0,0) \,, \\
Q_5 = (1,y,y) \,, &  Q_6 = (0,y,0) \,, & Q_7 =(x,1,x) \,, & Q_8 = (x,0,0)  \,.
\end{array}
\]
Considering the convex hull $\I$ of $Q_1, \ldots, Q_8$, the coherence of prevision assessment $\M=(x,y,z)$ on $\F$ requires that the condition $\M \in \I$ be satisfied, which amounts to solvability of the following system
\[\begin{array}{l}
(\S) \hspace{1 cm}
\M=\sum_{h=1}^8 \lambda_hQ_h,\;
\sum_{h=1}^8 \lambda_h=1,\; \lambda_h\geq 0,\, \forall \, h \,.
\end{array}
\]
We observe that
\[
\begin{array}{ll}
Q_5=yQ_1+(1-y)Q_2, & Q_6=yQ_3+(1-y)Q_4, \\ Q_7=xQ_1+(1-x)Q_3, & Q_8=xQ_2+(1-x)Q_4;
\end{array}\]
then, the convex hull $\I$  is the tetrahedron with vertices $Q_1,Q_2,Q_3,Q_4$. Thus, $(\S)$
is equivalent to the system
\[\begin{array}{l}
(S')
\hspace{1 cm}
\P=\sum_{h=1}^4 \lambda'_hQ_h,\;
\sum_{h=1}^4 \lambda_h'=1,\; \lambda_h'\geq 0,\, \forall \, h \,,
\end{array}
\]
with
\[
\begin{array}{ll}
\lambda_1'=\lambda_1+y\lambda_5+x\lambda_7, \;& \lambda_2'=\lambda_2+(1-y)\lambda_5+x\lambda_8,\\
\lambda_3'=\lambda_3+y\lambda_6+(1-x)\lambda_7, \;&
\lambda_4'=\lambda_4+(1-y)\lambda_6+(1-x)\lambda_8\,.
\end{array}\] \\
Then, $\M \in \I$ if and only if $(\S')$ is solvable. We observe that $(\S')$ can be written as
\[
(S') \hspace{0.5 cm}
\lambda_1'+\lambda_2'=x\,,\; \lambda_1'+\lambda_3'=y\,,\;
\lambda_1'=z\,,\;
\lambda_1'+\lambda_2'+\lambda_3'+\lambda_4'=1, \; \lambda_h'\geq 0,\, \forall \, h \,;
\]
%
that is
\[
(S') \hspace{0.5 cm}
\lambda_1' = z\,,\; \lambda_2' = x - z\,,\; \lambda_3' = y - z \,,\;
\lambda_4' = z - (x+y-1), \; \lambda_h'\geq 0,\, \forall \, h \,.
\]
As it can be easily verified, $(S')$ is solvable if and only if
\[
\max\{x+y-1,0\} =  z' \leq z \leq z'' = \min\{x,y\} \,.
\]
 Moreover, for each solution $(\lambda_1', \lambda_2', \lambda_3', \lambda_4')$ of $(\S')$ the vector  $(\lambda_1, \ldots, \lambda_8)=(\lambda_1', \lambda_2', \lambda_3', \lambda_4',0,0,0,0)$ is a solution of $(\S)$ such that
\[\begin{array}{l}
\sum_{r: C_r \subseteq H} \lambda_r = \sum_{r: C_r \subseteq K} \lambda_r = \sum_{r: C_r \subseteq H \vee K} \lambda_r = 1 > 0 \,,
\end{array}
\]
and hence $I_0 = \emptyset$; then, by Theorem \ref{CNES-PREV-I_0-INT}, the solvability of $(\S)$ is also sufficient for the coherence of $\M$. Therefore, the extension $\pr[(A|H) \wedge (B|K)] = z$ of the assessment $(x,y)$, with $(x,y) \in [0,1]^2$, is coherent if and only if
\begin{small}
\[
max \; \{P(A|H)+P(B|K)-1,0\}  \leq \pr[(A|H) \wedge (B|K)]  \leq  min \;   \{P(A|H),P(B|K)\}\,.
 \]
\end{small}
\end{proof}
We remark that for the quasi conjunction $\C(A|H,B|K)$ only holds the inequality on the lower bound; indeed, the extension $P[\C(A|H,B|K)] = \gamma$ of the assessment $(x,y)$ is coherent if and only if $\gamma' \leq \gamma \leq \gamma''$, where \linebreak
$\gamma' = z'=\max\{x+y-1,0\}$ and $\gamma'' = \frac{x+y-2xy}{1-xy}$ if $(x,y) \neq (1,1)$; $\gamma'' = 1$ if $(x,y) = (1,1)$. We observe that: $\gamma'' \geq max\{x,y\} \geq min\{x,y\} = z''$. \\
A probabilistic analysis of the lower and upper bounds for the quasi conjunction, in terms of t-norms and t-conorms, has been given in \cite{GiSa10,GiSa12c}.
\section{Negation and Disjunction}
Given any coherent assessment $(x,y,z)$ on $\{A|H, B|K, (A|H) \land (B|K)\}$, it holds that $(A|H) \land (B|K)\in \{1,0,x,y,z\} \subset [0,1]$. We recall that for conditional events the negation is usually defined as $(E|H)^c=E^c|H = (1-E)|H$. In our approach we have $(1-E)|H = 1 - E|H$; hence $(E|H)^c = 1 - E^c|H$. Then, for the conjunction of two conditional events, we give the following
\begin{definition}{\rm Given any conditional events $A|H, B|K$, the negation of the conjunction $(A|H) \land (B|K)$ is defined as
\[
[(A|H) \land (B|K)]^c = 1 - (A|H) \land (B|K) \,.
\]
}\end{definition}
We observe that
\[
\begin{array}{l}
[(A|H) \land (B|K)]^c = 1 - min \{A|H, B|K\} \,|\, (H \vee K) =\\
= (1 - min \{A|H, B|K\}) \,|\, (H \vee K) = max\{A^c|H, B^c|K\} \,|\, (H \vee K) \,.
\end{array}
\]
Then, based on the relation $A \vee B = (A^cB^c)^c$ (De Morgan's Law), for the disjunction of two conditional events we give the following
\begin{definition}{\rm Given any conditional events $A|H, B|K$, the disjunction $(A|H) \vee (B|K)$ is defined as: $(A|H) \vee (B|K) = [(A^c|H) \land (B^c|K)]^c$.
}\end{definition}
We observe that
\[
\begin{array}{l}
(A|H) \vee (B|K) = 1 - min \{A^c|H, B^c|K\} \,|\, (H \vee K) =\\
= (1 - min \{A^c|H, B^c|K\}) \,|\, (H \vee K) = max\{A|H, B|K\} \,|\, (H \vee K) \,.
\end{array}
\]
If we assess $P(A|H)=x, P(B|K)=y, \pr[(A|H) \vee (B|K)]=\gamma$, then
\[
(A|H) \vee (B|K)= 1 \cdot (AH \vee BK) + x \cdot H^cB^cK + y \cdot A^cHK^c + \gamma \cdot H^cK^c \,.
\]
{\em Prevision sum rule}. The classical formula $P(A \vee B) = P(A) + P(B) - P(AB)$ still holds for conjunction and disjunction of conditional events. In fact, by recalling (\ref{XgHvK}), we have
\begin{small}\[
(A|H) \vee (B|K) + (A|H) \wedge (B|K) = [min\{A|H, B|K\} + max\{A|H, B|K\}]|(H \vee K) =
\]
\[
= (A|H + B|K)|(H \vee K) = (A|H)|(H \vee K) + (B|K)|(H \vee K) = A|H+B|K\,.
\]
\end{small}
Then:
$\pr[(A|H) \vee (B|K)]=\pr(A|H)+\pr(B|K) -\pr[(A|H) \wedge (B|K)]$. Finally, assuming $A,H,B,K$ logically independent and defining $P(A|H)=x, P(B|K)=y, \pr[(A|H) \wedge (B|K)]=z, \pr[(A|H) \vee (B|K)]=\gamma$, it holds that $\gamma = x+y-z$ and from ($\ref{LOW-UPPER}$) we obtain $max  \{x,y\}  \leq z  \leq  min   \{x+y-1,1\}$, that is
\begin{small}\[
max  \{P(A|H),P(B|K)\}  \leq \pr[(A|H) \vee (B|K)]  \leq  min   \{P(A|H)+P(B|K)-1,1\} \,.
\]\end{small}
\section{Conclusions}
In this paper we have given some results on finite conditional random quantities and conditional prevision assessments in the setting of coherence. We have proposed a suitable representation for conditional random quantities which includes in particular the case of the (indicators of) conditional events represented as numerical quantities, with values 1, or 0, or $p$, where $p$ is the probability assessment for the given conditional event. By this representation we have given a meaning to the (iterated) conditional random quantity $(X|H)|K$ as a suitable conditional random quantity, by showing that $(X|H)|K \neq X|HK$, with $(X|H)|K = X|HK$ if $H \subseteq K$. Then, we have examined for two conditional events the logical operation of conjunction, its negation and by De Morgan's Law the associated disjunction. We recall that this problem has been largely studied by many authors in literature, especially in the field of artificial intelligence. Based on the recent paper by S. Kaufmann, we have shown that in the setting of coherence the logical operations can be defined in a natural way; moreover, the result of a conjunction or a disjunction of two conditional events is a conditional random quantity. In our paper we have obtained,  in a simple and direct way, the results on conjunction by Kaufmann and other general results. In particular we have determined the lower and upper bounds for the conjunction and disjunction of two conditional events, by showing that the classical properties valid for the conjunction and disjunction of two unconditional events continue to hold. In a future work we will deepen the study of conditional random quantities in the setting of coherence and we will analyze in general the operation of iterated conditioning given in \cite{GiSa13a}.

\bibliographystyle{sl}


\AuthorAdressEmail{Angelo Gilio}{Dipartimento di Scienze di Base e Applicate per l'Ingegneria\\
University of Rome "La Sapienza"\\
 Via A. Scarpa, 16 - 00161 Roma, Italy
}{angelo.gilio@sbai.uniroma1.it}
%
%
\AdditionalAuthorAddressEmail{Giuseppe Sanfilippo}{
University of Palermo\\
Viale delle Scienze, ed. 13 - 90128 Palermo, Italy
}
{giuseppe.sanfilippo@unipa.it}

\end{document}